\documentclass{amsart}
\usepackage{graphicx}
\usepackage{amssymb,amscd,amsthm,amsxtra}
\usepackage{latexsym}
\usepackage{epsfig,esint}
\usepackage{xcolor}
\newtheorem{thm}{Theorem}[section]
\newtheorem{cor}[thm]{Corollary}

\theoremstyle{definition}

\theoremstyle{remark}
\newtheorem{rem}[thm]{Remark}
\numberwithin{equation}{section}

\newcommand{\be}{\begin{equation}}
\newcommand{\ee}{\end{equation}}

\newcommand{\R}{\mathbb R}

\newcommand{\eps}{\varepsilon}

\newcommand{\p}{\partial}

\newcommand{\comment}[1]{}
\begin{document}

\title[Boundary Harnack Principle]{On the Boundary Harnack Principle for operators with different lower order terms}
\author{D. De Silva}
\address{Department of Mathematics, Barnard College, Columbia University, New York, NY 10027}
\email{\tt  desilva@math.columbia.edu}
\author{O. Savin}
\address{Department of Mathematics, Columbia University, New York, NY 10027}\email{\tt  savin@math.columbia.edu}

\begin{abstract}We provide the classical Boundary Harnack principle in Lipschitz domains for solutions to two different linear uniformly elliptic equations with the same principal part.
 \end{abstract}

%\thanks{}
%\subjclass{}%
%\keywords{One-phase free boundary problem; Harnack Inequality}
\maketitle
\begin{center}
{\it To our friend Sandro, who taught us about Boundary Harnack and much more.}
\end{center}

\section{Introduction}

The classical Boundary Harnack Principle (BHP) states that two positive harmonic functions that vanish on a portion of the boundary of a Lipschitz domain must be comparable up to a multiplicative constant. This statement appears first in works of Ancona, Dahlberg, Kemper, and Wu  respectively \cite{A,D,K,W}. In order to state it precisely and present its generalizations, we introduce some notation.
Given a Lipschitz function $g$, with $$g: B'_1 \subset \R^{n-1} \to \R, \quad g \in C^{0,1}, \quad \|g\|_{C^{0,1}} \leq L, \quad  g(0)=0,$$ 
we denote by $\Gamma$ its graph, 
$$\Gamma:= \{x_n =g(x')\},$$
and define the cylindrical region of radius and height $r$ above $\Gamma$ as
$$\mathcal C_r:=\{(x', x_n) \ : \ |x'| <r,  \  0 <  x_n -g(x') <  r\}.$$ 
The BHP states the following.
\begin{thm}\label{main_intro} Let $u,v>0$ solve $$\Delta u=\Delta v=0 \quad \text{in $\mathcal C_1$},$$ and vanish continuously on $\Gamma$. Then,\begin{equation}\label{BHI} C^{-1}\frac v u \left(\frac 12 e_n\right) \leq \frac v u \leq C \frac v u \left(\frac 12 e_n\right) \quad \text{in $\mathcal C_{1/2},$}\end{equation} with $C$ depending on $n$
and $L$. \end{thm}

Sandro Salsa played a central role in extending the BHP to uniformly elliptic operators in both divergence  and non-divergence form. The result was established first for operators in divergence form by Caffarelli, Fabes, Mortola and Salsa in \cite{CFMS}, while the case of operator in non-divergence form was treated by Fabes, Garofalo, Marin-Malave and Salsa in \cite{FGMS}. The BHP for operators in divergence form was extended also to more general NTA domains by Jerison and Kenig in \cite{JK}. The case of H\"older domains and operators in divergence form was  addressed with probabilistic techniques by Bass and Burdzy, and Banuelos, Bass and Burdzy in \cite{BB1,BBB} respectively. An analytic proof was then provided by Ferrari in \cite{F}. For H\"older domains and operators in non-divergence form, it is necessary that the domain is $C^{0,\alpha}$ with $\alpha>1/2$ (or that it satisfies a uniform density property) - a proof based on a probabilistic approach can be found in the work of Bass and Burdzy \cite{BB2}, while a subsequent analytic proof was given in \cite{KS} by Kim and Safonov.
In \cite{DS1} we provided a unified analytic proof of Theorem \ref{main} which holds for both operators in non-divergence and in divergence form and with lower order terms, and extends also to the case of H\"older domains \cite{DS2}. 

It is well known that the BHP can fail in the presence of a right hand side. The typical example to have in mind is when $\Gamma$ is a 2D cone of opening $\alpha$ and 
$\Delta u=0$, $\Delta v=-1$ in $\mathcal C_1$. Then the BHP only holds if the quadratic error induced by the right hand side is small compared to the growth of the harmonic function near the origin in $\mathcal C_1$. This means that the BHP holds if $\alpha > \pi/2$ and it fails otherwise.
The general case of equations with right hand sides in Lipschitz domains was considered by Allen and Shahgholian in \cite{AS} where a sharp criteria for the BHP to hold was established. This result was then extended to a larger class of operators in \cite{AKS}, by the same two authors and Kriventsov. In particular, the BHP holds for linear uniformly elliptic operators with right hand side in $L^\infty$ only if the Lipschitz constant $L$ of the domain is sufficiently small.

In this note, we provide versions of Theorem \ref{main_intro} for solutions to different uniformly elliptic homogeneous equations with the same principal part on $\Gamma$. Our results say that the BHP continues to hold under quite general assumptions on the operators and domain if the two linear operators differ only in the lower order terms. These results do not follow from the above mentioned works for inhomogeneous equations, just by treating the lower order terms as a right hand side.

The motivation for this type of result comes from the theory of nonlocal equations. After localizing a translation invariant kernel to a compact set, the remaining contributions from the far away interactions have the effect of introducing a lower zeroth order term whose $L^\infty$ coefficient depends on the solution itself. In this localized setting, two solutions solve two different equations that differ from each other in the lowest order term. Indeed, a nonlocal version of the results presented here is used in \cite{DSV} to establish the boundary regularity for graphical nonlocal minimal surfaces. 

\subsection{Main Results} For clarity of exposition, we present here the precise formulation of our result in the setting of operators with different zero order terms and with no first order terms.

Denote $$\mathcal{L}_{c} u : = tr(A(x)D^2u) + cu,$$ with $A$ uniformly elliptic, i.e
$$\lambda I \leq A \leq \Lambda I,$$ and $c \in L^\infty.$ Our main basic result for two different operators $\mathcal L_{c}$ and $\mathcal L_{c'}$ is stated below. Constants depending possibly on $n,\lambda, \Lambda, L$ and $\|c\|_{L^\infty}, \|c'\|_{L^\infty}$ are called universal.

\begin{thm}\label{main}
Let $u, v$ solve $$ \text{$\mathcal L_{c} u = 0=\mathcal L_{c'} v = 0$ in $\mathcal C_1$,}$$ vanish continuously on $\Gamma$, and let $u>0$. Then
$$\frac{v}{u} \in C^{0,\alpha}(\mathcal C_{1/2}),$$ for some $0<\alpha<1$ universal and \begin{equation}\label{bound1intro}\left \|\frac v u  \right \|_{C^{0,\alpha}(\mathcal C_{1/2})} \leq C \frac{\|v\|_{L^\infty(\mathcal C_1)}}{u\left(\frac 1 2 e_n\right)},\end{equation} with $C>0$ universal.  If $v>0$ then the ratio is bounded below in $\mathcal C_{1/2}$ and the right hand side in \eqref{bound1intro} can be replaced by $C \dfrac v u \left(\frac 12 e_n\right)$.\end{thm}

For simplicity, we assume throughout the paper that $A$ is continuous so that solutions are continuous up to the boundary. However, the constants involved in our estimates do not depend on the modulus of continuity of $A$.

The proof of the result above can be adapted to larger classes of operators. We refer to Section 3 for the precise statement of the corresponding results. In particular, Theorem \ref{main} continues to hold for linear operators with dependence on $\nabla u$, that is operators of the form 
$$\mathcal L_{{\bf b}, c} u := tr(A(x)D^2u) + {\bf b}(x) \cdot \nabla u + c(x)u,$$ with
${\bf b} \in L^\infty.$ In this setting, when $n \geq 3$, 
we assume further that $A(x)$ has small oscillation at small scales (Theorem \ref{main22}).
For equations with the same first order coefficient ${\bf b}$ this assumption is not necessary. We also establish a nontangential BHP version for the case of two different operators with second order H\"older coefficients that only coincide at a point on $\Gamma$ (Theorem \ref{main3}). In particular it shows that the boundary behavior of a positive solution to such a second order equation near a boundary point $x_0 \in \Gamma$ is obtained by freezing the second order coefficients of the operator at $x_0$, and by ignoring the lower order terms. Finally, we can extend Theorem \ref{main} to include equations in divergence form (Theorem \ref{main23} and Theorem \ref{main33}).
We remark that our results extend easily to NTA domains as well.

The paper is organized as follows. Section 2 is dedicated to the proof of our main Theorem \ref{main}. The next section provides the adaptation of the proof to other classes of operators. 

\section{The proof of Theorem \ref{main}}

In this section we provide the proof of Theorem \ref{main}. As mentioned in the Introduction, constants depending possibly on $n,\lambda, \Lambda, L$ and $\|c\|_{L^\infty}, \|c'\|_{L^\infty}$ will be called universal. 

For $a,b>0$, we use the notation $$a\sim b $$ to denote the inequality $$C^{-1} b \leq a \leq Cb, \quad \quad \text{$C>0$ universal.}$$ Universal constants may change from line to line in the body of the proof (here and in the remaining of the paper).

\begin{proof} Assume after multiplication by a constant that $u\left(\frac 1 2 e_n\right)=1,$ and let $u_0$ be a positive solution to 
$$\mathcal L_0 u_0 = 0 \quad \text{in $\mathcal C_1$}, \quad 
u_0\left(\frac 1 2 e_n\right)=1, $$ which vanishes on $\Gamma.$ In order to prove our result, 
it suffices to show that $\dfrac{u}{u_0}$ is a $C^{0, \alpha}$ function bounded below by a universal positive constant, since we can write 
$$\frac v u = \frac{v}{u_0} \cdot \frac{u_0}{u} \  \ \ .$$ Then $v$ can be written as
$$\mbox{$v = V^+-V^-$, with $\mathcal L_{c} V^\pm=0$ in $\mathcal C_1$ and $V^\pm = v^\pm$ on $\p \mathcal C_1.$}$$ Here $v^\pm$ denote the positive and negative part of $v$, that is $$\mbox{$v^+:= \max\{v,0\}$ and $v^-:= -\min\{v,0\}$.}$$ Thus, the ratio $\dfrac{V^\pm}{u_0}$ (after normalization) will also be a $C^{0,\alpha}$ function bounded below by a positive universal constant.

\smallskip

We prove now that the ratio $\dfrac{u}{u_0}$ is H\"older continuous.
Denote $r_k:= 2^{-k}$, $k=1,2,\ldots$ and let $w_k$ be the solution to 
$$\mathcal L_0 w_k =0 \quad \text{in $\mathcal C_{r_k}$,}$$
$$w_k = u \quad \text{on $\p \mathcal C_{r_k}.$}$$ 

We divide our proof in three steps.

\medskip

\noindent{\it Claim 1.}  For all $k \geq k_0$ universal, \begin{equation}\label{bound2} |w_k - u| \leq C r_k^\beta \ w_k\left(r_k \frac{e_n}{2}\right) \quad \text{in $\mathcal C_{r_k}$,} \end{equation}
for some $C, \beta>0$ universal.

\begin{proof}After a dilation, 
$$\tilde w(x) : = w_k(r_k x), \quad \tilde u(x):= u(r_k x),$$
we have that 
$$\tilde{\mathcal L_0} \tilde w=0, \quad \tilde{\mathcal L}_{\tilde c} \tilde u:= \tilde{\mathcal L_0} \tilde u + \tilde c \tilde u =0 \quad \text{in $\tilde {\mathcal{C}}_1$}$$
with 
$$\tilde c (x) : = r_k^2 c(r_k x), \quad \tilde{\mathcal L}_0 w: = tr(A(r_k x)D^2 w)  , \quad \tilde {\mathcal{C}}_1 := r_k^{-1} \mathcal C_1.$$
Since $\tilde w - \tilde u$ solves $$\tilde{\mathcal L}_0 (\tilde w - \tilde u) = \tilde f, \quad \tilde f:=\tilde c \tilde u, \quad \text{in $\tilde {\mathcal{C}}_1$}$$ and vanishes on the boundary, by the maximum principle, 
\be\label{infinity} |\tilde w - \tilde u| \leq C \|\tilde f\|_{L^\infty} \leq C\|\tilde c\|_{L^\infty}\|\tilde u\|_{L^\infty} \quad \text{in $\tilde {\mathcal{C}}_1$}.\ee
Using the definition of $\tilde c$, and by applying the classical Carleson estimate for the $ \tilde{\mathcal L}_{\tilde c}$ operator in $\tilde {\mathcal{C}}_2$ (notice that $\tilde u$ is defined and solves the equation in the larger cylinder  $\tilde {\mathcal{C}}_2$),
$$\|\tilde u\|_{L^\infty(\tilde {\mathcal{C}}_1)}\leq C \tilde u \left(\frac 1 2 e_n\right),$$
 we get that
$$\|\tilde w - \tilde u\|_{L^\infty(\tilde {\mathcal{C}}_1)} \leq C r_k^2 \tilde u \left (\frac 1 2 e_n\right).$$
Thus, for $k$ sufficiently large, $$C r_k^2 \leq \frac 1 2,$$ which gives
$$\tilde w\left(\frac 1 2 e_n\right) \sim \tilde u\left (\frac 1 2 e_n\right),$$
from which the desired conclusion follows with $\beta=2$.\end{proof}

\medskip

We proceed with the statement and proof of the next claim.

\medskip

{\noindent {\it Claim 2}.} \begin{equation}\label{bound4}
\left |\frac{w_k}{w_{k+1}} - 1\right |\leq C r_k^\beta, \quad \text{in $\mathcal C_{r_{k+2}}.$}
\end{equation}

\begin{proof} The argument in the proof of {\it{Claim 1}} above gives that (unraveling the scaling)
$$w_k\left(r_k \frac 1 2 e_n\right) \sim u\left (r_ k\frac 1 2 e_n\right),$$
and 
$$w_{k+1}\left(r_{k+1} \frac 1 2 e_n\right) \sim u\left(r_ {k+1}\frac 1 2 e_n\right).$$
On the other hand, by Harnack inequality, $$ u\left(r_ k\frac 1 2 e_n\right) \sim  u\left(r_ {k+1}\frac 1 2 e_n\right),$$ hence
$$ w_k\left(r_ k\frac 1 2 e_n\right) \sim  w_{k+1}\left(r_ {k+1}\frac 1 2 e_n\right).$$ Combining this fact with 
\eqref{bound2} applied to $w_k$ and $w_{k+1}$ we obtain that
\begin{equation}\label{bound3}
|w_{k+1} - w_k| \leq Cr_k^\beta w_{k+1}\left(r_ {k+1}\frac 1 2 e_n\right), \quad \text{in $\mathcal C_{r_{k+1}}.$}
\end{equation}
Thus, the classical boundary Harnack inequality for the operator $\mathcal L_0$ and the pair of functions $w_{k+1}-w_k$ and $w_{k+1}$ gives the desired claim \eqref{bound4}.\end{proof}

We are now ready for the next step.

\medskip

{\noindent {\it Claim 3.}} There exists $a_k \sim 1$ and $\delta>0$ small universal, such that 
\begin{equation}\label{bound5} \left| \frac{a_k u_0}{w_k} - 1\right| \leq C r_k^{\delta} \quad \text{in $\mathcal{C}_{r_{k+1}.}$}
\end{equation}

\begin{proof}
From \eqref{bound4}, and $m=0, 1, \ldots, k-1,$
\be\label{D}\frac{w_{k-m-1}}{w_{k-m}} = 1+ O (r_{k-m}^\beta), \quad \text{in $ \mathcal C_{r_{k+1}}.$}\ee
On the other hand, by boundary Harnack applied to $w_{k-m-1}$ and $w_{k-m}$ in $\mathcal C_{r_{k-m+1}}$ we have that for $\alpha>0$ small universal, 
\be \label{bound6}
\frac{w_{k-m-1}}{w_{k-m}}= b_{k-m} (1+ O(r_m^\alpha))  \quad \text{in $\mathcal C_{r_{k+1}},$}
\ee
with 
\be\label{b} b_{k-m} = 1 + O(r_{k-m}^\beta), \quad  b_{k-m} \sim 1,
\ee
in view of \eqref{D}.
Similarly, 
\be \label{bound7}
\frac{u_0}{w_0}= b_0(1+ O(r_k^\alpha)) \quad \text{in $\mathcal C_{r_{k+1}},$}  \quad b_0 \sim 1.
\ee
Thus, using \eqref{bound2} for $m=0, \ldots, [\frac k 2]$ and \eqref{bound6}-\eqref{bound7} for the remaining terms, we get 
$$
\frac{u_0}{w_k} = \left(\prod \frac{w_{k-m+1}}{w_{k - m}}\right) \frac{u_0}{w_0}= \left (\prod_{i=0}^{[\frac k 2 ]} b_i \right)(1+O(r_k^{\alpha/2}))(1+O(r_k^{\beta/2})).
$$ Here we used that 
\be\label{log}\prod_{i} (1+ \xi_i) = e^{\sum_{i} \ln(1+\xi_i)}=e^{O(\sum_i |\xi_i|)}, \quad |\xi_i|\leq \frac 1 2.\ee
This gives the desired claim for an appropriate $\delta$ and 
$$a_k := \left (\prod_{i=0}^{[\frac k 2 ]} b_i \right).$$
Notice that $a_k \sim 1$ in view of \eqref{b} and \eqref{log}.
\end{proof}

We are now ready to conclude the proof.
Combining \eqref{bound2} with \eqref{bound5} and using the standard Carleson estimate for $w_k$, we get that
$$|a_k u_0 - u| \leq Cr_k^\delta  \ a_k u_0\left(r_k \frac{e_n}{2}\right) \quad \text{in $\mathcal C_{r_{k+1}}$}.$$

Set $$M_k:= a_k u_0\left(r_k \frac{e_n}{2}\right),$$ and 
$$\tilde u_0(x) : = \frac{a_k}{M_k} u_0(r_k x), \quad \tilde u(x): = \frac{1}{M_k} u(r_k x), $$
then the inequality above reads
\be |\tilde u_0 - \tilde u| \leq C r_k^\delta \quad \text{in $\tilde{\mathcal C}_{1/2},$} \quad \quad \tilde u_0\left(\frac 1 2 e_n\right)=1. \ee
Using H\"older estimates for the operator $\tilde {\mathcal L}_0$ applied to $\tilde u_0 - \tilde u$, we obtain that for $\alpha$ small universal, and $\rho$ sufficiently small universal, 
\be\label{0gamma} \|\tilde u_0 - \tilde u\|_{C^{0,\alpha}(B_{\rho}(\frac 1 4 e_n))} \leq  C(\|\tilde u_0 - \tilde u\|_{L^\infty(\tilde{\mathcal C}_{1/2})} + \|\tilde f\|_{L^\infty(\tilde{\mathcal C}_{1/2})}) \leq Cr_k^\delta. \ee
On the other hand, since by Harnack inequality $\tilde u_0$ is bounded below and has bounded $C^{0,\alpha}$ norm in $B_{\rho}(\frac 1 4 e_n)$, we conclude that
$$\left\|\frac{\tilde u}{\tilde u_0} -1\right\|_{C^{0,\alpha}(B_{\rho}(\frac 1 4 e_n))} \leq Cr_k^\alpha,$$ where without loss of generality we assumed that $\alpha \leq \delta.$
After rescaling, this gives that 
$$\left\|\frac{u}{ a_k u_0} -1\right\|_{C^{0,\alpha}(B_{\rho r_k}(\frac{r_k}{4} e_n))} \leq C$$
and since $a_k$ is bounded by universal constants above and below, we conclude that 
$$ \left\|\frac{u}{ u_0}\right\|_{C^{0,\alpha}(B_{\rho r_k}(\frac{r_k}{4} e_n))} \leq C.$$
This gives the desired result by a standard covering argument. 
\end{proof}

\section{Some variants}

In this section we extend Theorem \ref{main} to other classes of linear uniformly elliptic operators.

\subsection{Linear Operators with $\nabla u$ dependence.} Let 
$$\mathcal L_{{\bf b}, c} u := tr(A(x)D^2u) + {\bf b}(x) \cdot \nabla u + c(x)u,$$ with
$$\lambda I \leq A \leq \Lambda I, \quad {\bf b}, c \in L^\infty.$$ In this case, the arguments in the proof of Theorem \ref{main} continue to hold as long as we have interior gradient estimates. For this reason, in dimension $n \geq 3$, 
we assume further that $A(x)$ has small oscillation at small scales, i.e.
\be\label{osc} |A(x)-A(y)|\leq \eps_0, \quad \text{for $|x-y| \leq r_0$}\ee
for some $\eps_0=\eps_0(n, \lambda, \Lambda)$ and some $r_0>0.$ Given two different operators, $\mathcal L_{{\bf b},c}$ and $\mathcal L_{{\bf b'},c'}$, in this setting universal constants will depend possibly on $\|{\bf b}\|_{L^\infty}, \|{\bf b'}\|_{L^\infty}$ and $r_0$ as well. For clarity of exposition, we restate the BHP for this class of operators.
\begin{thm}\label{main22}
Let $A$ satisfy \eqref{osc} (if $n\geq 3$), and let $u,v$ solve
$$ \mathcal L_{{\bf b}, c} u = \mathcal L_{{\bf {b'}},c'} v = 0 \quad \text{in $\mathcal C_1$,}$$ 
vanish continuously on $\Gamma$, and $u>0$. 
Then
$$\frac{v}{u} \in C^{0,\alpha}(\mathcal C_{1/2}),$$ for some $0<\alpha<1$ universal and \begin{equation}\label{bound13}\left \|\frac v u  \right \|_{C^{0,\alpha}(\mathcal C_{1/2})} \leq C \frac{\|v\|_{L^\infty(\mathcal C_1)}}{u\left(\frac 1 2 e_n\right)},\end{equation} with $C>0$ universal. If $v>0$ then the ratio is bounded below in $\mathcal C_{1/2}$ and the right hand side in \eqref{bound13} can be replaced by $C \dfrac v u \left(\frac 12 e_n\right)$.
\end{thm}

\begin{rem} In Theorem \ref{main22}, if $b=b'$, then the assumption \eqref{osc} is no longer necessary as the proof of Theorem \ref{main} can be applied directly in this context.
\end{rem}
\begin{proof} In order to apply the strategy in Theorem \ref{main}, it suffices to verify that \eqref{infinity} and \eqref{0gamma} continue to hold in this setting, with $w$ the replacement of $u$ for the unperturbed operator. 
As before, we treat the lower order term ${\bf b} \cdot \nabla u + c u$ as a perturbation. However, since gradient estimates do not hold up to the Lipschitz boundary, we need to combine H\"older boundary estimates with interior gradient estimates. Below are the details.
After a dilation and multiplication by a constant we may assume that 
\be\label{res} |A(x)- I| \leq \eps_0(n), \quad |{\bf b}(x)| \leq C r, \quad |c(x)| \leq C r^2,\ee for some $r$ sufficiently small, 
$$ \mathcal L_{{\bf b}, c}u=0,  \quad u> 0 \quad \text{in $\mathcal C_2$}, \quad u\left(\frac  1 2 e_n\right)=1.
$$ By Carleson estimates we obtain that $u$ is bounded in $L^\infty$ in $\mathcal C_{3/2}$ by a universal constant. Then, H\"older estimates up to the boundary applied in $\mathcal C_{3/2}$ give that \be\label{eta1} \|u\|_{C^{0,\eta}(\mathcal C_1)} \leq C\ee for some $\eta>0$ small and $C>0$ universal.
Let $w$ be the replacement of $u$ corresponding to the unperturbed operator 
$$\mathcal L_{{\bf 0},0} w=0 \quad \text{in $\mathcal C_1$} \quad w=u \quad \text{on $\p \mathcal C_1$.}$$
The arguments of the proof of Theorem \ref{main} remain valid if we show that for $\beta>0$ small universal,
\be\label{inf22} \|u-w\|_{L^{\infty}(\mathcal C_{1})} \leq r^{\beta}\ee
and for some $\alpha, \rho>0$ small universal 
\be\label{0gamma22} \|u-w\|_{C^{0,\alpha}(B_\rho(\frac 1 4 e_n))} \leq r^{\beta}.\ee
By H\"older estimates up to the boundary, in view of \eqref{eta1} (up to possibly renaming $\eta$),
\be\label{eta12}  \|w\|_{C^{0,\eta}(\mathcal C_1)}  \leq C.\ee
Set,
$$D' := \{x \in \mathcal C_1 \ : \ dist(x, \p \mathcal C_1) > r^\delta\}$$
for some $\delta>0$ small.
Using that $u=w$ on the boundary of $\mathcal C_1$, and $\|u-w\|_{C^{0,\eta}(\mathcal C_1)} \leq C$ (by \eqref{eta1}-\eqref{eta12}), we ge that
\be \|u-w\|_{L^\infty} \leq C r^{\delta \eta} \quad \text{in $\mathcal C_1 \setminus D'.$}\ee
On the other hand, 
\be\label{00}\mathcal L_{{\bf 0}, 0} (u-w) = f \quad \text{in $D'$}\ee
with (see \eqref{res})
$$|f| = |{\bf b} \cdot \nabla u + c u| \leq C(r \|\nabla u\|_{L^\infty(D')} + r^2\|u\|_{L^\infty(D')}).$$
Since at each point $x \in D'$, $B_{r^\delta}(x) \subset \mathcal C_1$, by interior gradient estimates (Cordes-Nirenberg if $n \geq 3$) for the equation satisfied by $u$, we find that
$$ |\nabla u(x)| \leq C r^{-\delta}\|u\|_{L^\infty(\mathcal C_1)} \leq C r^{-\delta}, \quad \text{in $D'$.}$$
In conclusion,
$$|f| \leq C r^{1-\delta}, \quad \text{in $D'$.}$$
We deduce that
$$ |u-w| \leq C (r^{\delta \eta} + r^{1-\delta}) \quad \text{in $D'$,} $$ from which \eqref{inf22} follows.
Moreover, by the interior H\"older estimates for \eqref{00}, we also obtain \eqref{0gamma22}.
\end{proof}

\begin{rem} Theorem \ref{main22} continues to hold if the lower order coefficients, ${\bf b}, c$ satisfy 
$$|{\bf b}| \leq C dist(x, \Gamma)^{\delta -1}, \quad |c| \leq C dist(x, \Gamma)^{\delta -2}$$ for some $\delta>0$ small. The reason for that is that the coefficients of operator $\mathcal L_{{\bf b}, c}$ are improving after dilation, and the boundary H\"older estimates are still valid under this growth assumption.
\end{rem}

\subsection{Linear Operators with second order H\"older coefficients.} Let 
$$\mathcal L_{A,{\bf b}, c} u := tr(A(x)D^2u) + {\bf b}(x) \cdot \nabla u + c(x)u,$$ with
$$\lambda I \leq A \leq \Lambda I, \quad A \in C^{0,\gamma}, \quad {\bf b}, c \in L^\infty.$$

In this setting, universal constants associated to different operators, $\mathcal L_{A,{\bf b}, c}$ and $\mathcal L_{A',{\bf b'}, c'},$ will depend also on $[A]_{C^{0,\gamma}}, [A']_{C^{0,\gamma}}.$ The BHP takes the following form.

\begin{thm}\label{main3}
Let $u,v$ solve
$$\mathcal L_{A,{\bf b}, c} u = \mathcal L_{A',{\bf {b'}},c'} v = 0 \quad \text{in $\mathcal C_1$}, $$ vanish continuously on $\Gamma$, and $u>0.$ Assume that
$A(0)=A'(0).$ Then $\dfrac v u$ is $C^{0,\alpha}$ in the non-tangential cone $$\mathcal T:=\{x_n \geq 2L|x'|\} \cap \mathcal C_{1/2},$$ for some $0<\alpha< 1$ universal and 
\be\label{bound144}\left\|\frac v u\right\|_{C^{0,\alpha}(\mathcal T)} \leq  C \frac{\|v\|_{L^\infty}(\mathcal C_1)}{u(\frac 1 2 e_n)},\ee
with $C>0$  universal. If $v>0$ then the ratio is bounded below and the right hand side in \eqref{bound144} can be replaced by $C \dfrac v u \left(\frac 12 e_n\right)$.\end{thm}
\begin{proof} The proof follows almost verbatim as in Theorem \ref{main22}, up to controlling the right hand side $f$ of the equation satisfied by $u$ minus its harmonic replacement $w$. Thus, after a dilation and multiplication by a constant 
\be |A(x)- I| \leq Cr^\gamma, \quad |{\bf b}(x)| \leq C r, \quad |c(x)| \leq C r^2,\ee for some $r$ sufficiently small, 
$$ \mathcal L_{A,{\bf b}, c}u=0,  \quad u> 0 \quad \text{in $\mathcal C_2$}, \quad u\left(\frac  1 2 e_n\right)=1,
$$ and for some $\eta>0$ small universal,
\be\label{ET} \|u\|_{C^{0,\eta}(\mathcal C_1)} \leq C. \ee 
Let $w$ be the harmonic replacement of $u$ $$ \Delta w=0 \quad \text{in $\mathcal C_1,$} \quad
 w=u \quad \text{on $\p \mathcal C_1$.}$$ Then, by H\"older estimates up to the boundary, in view of \eqref{ET}, 
\be\label{TT} \|u - w\|_{C^{0,\eta}(\mathcal C_1)}  \leq C.\ee
As before, it suffices to show that \eqref{inf22} and \eqref{0gamma22} hold for $\alpha, \beta, \rho>0$ small universal. Set,
$$D' := \{x \in B_1 \cap D \ : \ dist(x, \p \mathcal C_1) > r^\delta\}$$
for some $\delta>0$ small.
Since $u=w$ on the boundary of $\mathcal C_1$ and \eqref{TT} holds, we ge that
$$ \|u-w\|_{L^\infty} \leq C r^{\delta \eta} \quad \text{in $\mathcal C_1 \setminus D'.$}$$
On the other hand, 
$$\mathcal L_{A,{\bf b}, c} (u-w) = f \quad \text{in $D'$}$$
with
\begin{align*}|f| &= |tr((A-I)D^2 w)+{\bf b} \cdot \nabla w + c w|\\
\ & \leq C(r^\gamma \|D^2w\|_{L^\infty(D')} + r \|\nabla w\|_{L^\infty(D')} + r^2\|w\|_{L^\infty(D')}).\end{align*}
Hence, since $w$ is harmonic and
$$\|D^2w\|_{L^\infty(D')} \leq C r^{-2\delta},$$
we get
$$|f| \leq C r^{\gamma -2 \delta} $$
and we conclude as in the previous proof.
\end{proof}

\begin{cor}\label{coro} Under the assumptions of Theorem \ref{main3}, if $$A=A' \quad \text{on $\Gamma$}$$ then the conclusion holds in the full domain $\mathcal C_{1/2}$ as for Theorem \ref{main22}.
\end{cor}

\subsection{Linear operators in divergence form.}  Let 
$$L_{{\bf b}, c} u := div(A(x)\nabla u) + {\bf b}(x) \cdot \nabla u + c(x)u,$$ with
$$\lambda I \leq A \leq \Lambda I, \quad {\bf b} \in L^p, c \in L^{\frac p 2}, \quad p>n.$$ 
In this setting, universal constants associated to different operators, $L_{{\bf b}, c}$ and $L_{{\bf b'}, c'},$ will depend on $n, \lambda, \Lambda, L, \|b\|_{L^p}, \|b'\|_{L^p}, \|c\|_{L^{p/2}}, \|c'\|_{L^{p/2}}$.

\begin{thm}\label{main23}
Let $u,v$ solve
$$ L_{{\bf b}, c} u = L_{{\bf {b'}},c'} v = 0 \quad \text{in $\mathcal C_1$,}$$ 
vanish continuously on $\Gamma$, and $u>0$. 
Then the conclusion of Theorem \ref{main22} holds.
\end{thm}
\begin{proof} The proof is the same as the proof of Theorem \ref{main22}, but we need to use energy methods to estimate $f$.
As usual, we may assume that 
\be\label{resc} \|{\bf b}\|_{L^p(\mathcal C_2)} \leq C r^\delta, \quad \|c\|_{L^{p/2}(\mathcal C_2)} \leq C r^\delta,\ee for some $r$ sufficiently small, and an appropriate $\delta$ universal,
$$ L_{{\bf b}, c}u=0,  \quad u> 0 \quad \text{in $\mathcal C_2$}, \quad u\left(\frac  1 2 e_n\right)=1,$$ with
\be\label{eta111} \|u\|_{C^{0,\eta}} \leq C \quad \text{on $\mathcal C_1$,}\ee for some $\eta>0$ small and $C>0$ universal.
Let $w$ be the replacement of $u$ corresponding to the unperturbed operator 
$$L_{{\bf 0},0} w=0 \quad \text{in $\mathcal C_1$,} \quad w=u \quad \text{on $\p \mathcal C_1$.}$$
By H\"older estimates up to the boundary, in view of \eqref{eta111}, 
\be\label{eta20}  \|u-w\|_{C^{0,\eta}(\mathcal C_1)}  \leq C.\ee
As before, it suffices to show that \eqref{inf22} and \eqref{0gamma22} hold for $\alpha, \beta, \rho>0$ small universal. We will prove the stronger claim that
\be\label{conclusion} \|u-w\|_{C^{0,\alpha}(\mathcal C_{1})} \leq r^{\beta}.\ee
We have, 
\be\label{000} L_{{\bf 0}, 0} (u-w) = f \quad \text{in $\mathcal C_1$}\ee
with $u-w=0$ on the boundary and 
$$f = -({\bf b} \cdot \nabla u + c u).$$ Since by Caccioppoli inequality and Carleson estimates,  $$\|\nabla u\|_{L^2(\mathcal C_1)}, \|u\|_{L^\infty(\mathcal C_1)} \leq C,$$
H\"older's inequality gives that $f \in L^q(\mathcal C_1)$ with $\frac 1 q = \frac 1 2 + \frac 1 p,$ and by \eqref{resc}
$$\|f\|_{L^q(\mathcal C_1)} \leq C r^\delta.$$
Multiplying \eqref{000} by $u-w$, integrating by parts, and using ellipticity, we get that,
$$\int_{\mathcal C_1} |\nabla (u-w)|^2 dx \leq C\int_{\mathcal C_1} f (w-u) dx \leq C\|f\|_{L^q(\mathcal C_1)} \|w-u\|_{L^{q'}(\mathcal C_1)}.$$ By Sobolev embedding, since $q' < 2^*$ ($p>n$), we conclude that
$$\|w-u\|_{H^1(\mathcal C_1)} \leq C r^\delta.$$
Now, the desired conclusion \eqref{conclusion} follows by interpolating the inequality above and \eqref{eta20}.
\end{proof}
Finally, let 
$$L_{A,{\bf b}, c} u := tr(A(x)D^2u) + {\bf b}(x) \cdot \nabla u + c(x)u,$$ with
$$\lambda I \leq A \leq \Lambda I, \quad A \in C^{0,\gamma}, \quad {\bf b} \in L^p, c \in L^{\frac p 2}, \quad p>n.$$

In this setting, universal constants associated to different operators, $L_{A,{\bf b}, c}$ and $L_{A',{\bf b'}, c'},$ will depend also on $[A]_{C^{0,\gamma}}, [A']_{C^{0,\gamma}}.$ The BHP takes the following form. 

\begin{thm}\label{main33}
Let $u,v$ solve
$$L_{A,{\bf b}, c} u = L_{A',{\bf {b'}},c'} v = 0 \quad \text{in $\mathcal C_1$}, $$ vanish continuously on $\Gamma$, and $u>0.$ Assume that
$A(0)=A'(0).$ Then the conclusion of Theorem \ref{main3} holds. Moreover, if $A=A'$ on $\Gamma$, the conclusion of Theorem \ref{main22} holds.\end{thm}
\begin{proof}We apply the same proof as above, with $w$ being the harmonic replacement of $u$ in $\mathcal C_1$, and
\be\label{resca} \|A-I\|_{L^\infty} \leq C r^\gamma, \quad \|{\bf b}\|_{L^p} \leq C r^\delta, \quad \|c\|_{L^{p/2}(\mathcal C_2)} \leq C r^\delta.\ee 
Then 
$$f =  -div((A-I)\nabla u)+ g, \quad g:=-{\bf b} \cdot \nabla u + c u,$$ and the integration by parts gives,
$$\int_{\mathcal C_1} |\nabla (u-w)|^2 dx \leq C_0(\|g\|_{L^q(\mathcal C_1)} \|w-u\|_{L^{q'}(\mathcal C_1)} + E),$$
with 
$$E:=\int_{\mathcal C_1} (A-I)\nabla u \nabla(u-w)) \ dx.$$ To handle this extra term, we use H\"older's and Young's inequality and conclude that
$$E \leq r^\gamma(\frac {C_0}{2}\|\nabla u\|^2_{L^2(\mathcal C_1)}+\frac{1}{2C_0}\|\nabla (u-w)\|^2_{L^2(\mathcal C_1)}).$$
By Caccioppoli and Carleson estimates $$\|\nabla u\|_{L^2(\mathcal C_1)} \leq C \|u\|_{L^2(\mathcal C_1)} \leq C.$$
Thus,
$$\|w-u\|_{H^1(\mathcal C_1)} \leq C (r^\delta + r^\gamma),$$ from which the result follows as before.
\end{proof}

\end{document}